\documentclass[12pt]{article}
\usepackage{enumerate,amssymb}
\newtheorem{theorem}{Theorem}[section]
\newtheorem{lemma}[theorem]{Lemma}

\newtheorem{corollary}[theorem]{Corollary}

\newenvironment{proof}[1][Proof]{\begin{trivlist}
\item[\hskip \labelsep {\bfseries #1}]}{\end{trivlist}}
\newenvironment{definition}[1][Definition]{\begin{trivlist}
\item[\hskip \labelsep {\bfseries #1}]}{\end{trivlist}}

\def\qed{{\hfill\hphantom{.}\nobreak\hfill$\blacksquare$}}
\title{Characterizations of finite classical polar spaces by intersection numbers with hyperplanes and spaces of codimension 2}

\author{S. De Winter and J. Schillewaert\\ Department of Pure Mathematics and Computer
Algebra, Ghent University\\ 
Krijgslaan 281 S 22, B-9000 Gent, Belgium\\
sgdwinte,jschille@cage.ugent.be}
\begin{document}







\date{September 28, 2007}
\maketitle
\abstract{
In this article we show that non-singular quadrics and non-singular Hermitian varieties are completely characterized by their intersection numbers with respect to hyperplanes and spaces of codimension 2. 
This strongly generalizes a result by Ferri and Tallini \cite{FT} and also provides necessary and sufficient conditions for quasi-quadrics ( respectively their Hermitian analogues) to be non-singular quadrics (respectively Hermitian varieties).}
\section{Introduction}
When Segre \cite{Segre} proved his celebrated characterization of conics (``every set of $q+1$ points in $\mathrm{PG}(2,q)$, $q$ odd, no three 
of which are collinear, is a conic''), he did more than proving a beautiful and interesting theorem; he in fact provided the starting point of a new 
direction in combinatorial geometry. In this branch of combinatorics the idea is to provide purely combinatorial characterizations of objects classically
defined in an algebraic way. This article wants to contribute to this theory by proving strong characterizations of classical finite polar spaces. 

\subsection{Polar spaces}

Though we suppose the reader to be familiar with polar spaces, that is, via polarities, quadrics and Hermitian varieties in $\mathrm{PG}(n,q)$,
we will quickly overview notation and terminology and recall some results that will be useful throughout this paper.

A finite classical non-singular polar space is one of the following:

\begin{itemize}
\item a hyperbolic quadric in $\mathrm{PG}(2n+1,q)$, denoted by $Q^+(2n+1,q)$;
\item a parabolic quadric in $\mathrm{PG}(2n,q)$, denoted by $Q(2n,q)$;
\item an elliptic quadric in $\mathrm{PG}(2n+1,q)$, denoted by $Q^-(2n+1,q)$;
\item a Hermitian variety in $\mathrm{PG}(n,q^2)$, denoted by $H(n,q^2)$;
\item a symplectic polar space in $\mathrm{PG}(2n+1,q)$, denoted by $W_{2n+1}(q)$.
\end{itemize}

{\bf Remark.} Throughout this article we will use the above notations for non-singular polar spaces. A cone with vertex a point $p$ or a line $L$ over a non-singular polar space, e.g. over a $Q(2n,q)$ will be denoted as $pQ(2n,q)$,
respectively $LQ(2n,q)$ (a small letter will always indicate the vertex of the cone is a point, while a capital letter will indicate the vertex is a line). 
With these conventions, the notation will always immediately tell whether the considered polar space is singular or non-singular. Only in the statements of our lemmas and theorems we will explicitly mention the (non-)singular character of the considered polar spaces.  

\subsection{Characterizations of polar spaces}

A {\it thick partial linear space of order $(s,t)$}, is a point-line geometry such that each line contains $s+1>2$ points, such that each point is contained in $t+1>2$ lines, and such that two distinct lines intersect in at most one point.

A thick partial linear space $S$ is said to be a {\it Shult space} if for every antiflag $(p,L)$ of $S$ either exactly $1$ or all points of $L$ are collinear with $p$. 

The following very nice theorem is the result of work done by Veldkamp \cite{Veldkamp59}, Tits \cite{Tits74}, Buekenhout and Shult \cite{Buekenhout-Shult74}.

\begin{theorem}\label{Shult}
Suppose that $S$ is a Shult space such that no point of $S$ is collinear with all other points of $S$, and such that there exists an antiflag $(p,L)$ with the property that $p$ is collinear with all points of $L$. Then $S$ is isomorphic to the point-line geometry of a 
finite classical non-singular polar space. 
\end{theorem}

If a Shult space is fully embedded in a projective space then the following theorem follows from Buekenhout and Lef\`evre \cite{BL}, and Lef\'evre-Percsy \cite{L1,L2}.

\begin{theorem}\label{BL}
Suppose $S$ is a Shult space  such that no point of $S$ is collinear with all other points of $S$. If $S$ is fully embedded in a projective space, then $S$ consists of the points and lines of a finite classical polar space. 
(Here fully embedded means that the set of lines of $S$ is a subset of the set of lines of the projective space and that the point set of $S$ is the set of all points contained in these lines.)
\end{theorem}

The following characterizations of quadrics and Hermitian varieties, which can be found in Hirschfeld and Thas \cite{HT}, will be of great 
importance in the proof of our main theorem. These theorems also provide nice examples of Segre-type theorems.  
\begin{definition}
A point set $K$ in $PG(n,q)$ is said to be of {\em type $(r_1,r_2,\cdots,r_s)$} if  $|L\cap K|\in\{r_1,r_{2},\cdots ,r_{s}\}$ for all lines $L$ of $PG(n,q)$. A point $p\in K$ 
is called {\em singular} with respect to $K$ if all lines through $p$ intersect $K$ either in 1 or in $q+1$ points. If a set $K$ contains a singular point,
then $K$ is called {\em singular}.
\end{definition}
The theorem below is an amalgation of results by Tallini-Scafati \cite{TS}, Hirschfeld and Thas \cite{HT80} and Glynn \cite{G}.
\begin{theorem}\label{Hermitian}
Let $K$ be a non-singular point set of type $(1,r,q^2+1)$ in $\mathrm{PG}(n,q^2)$, $n\geq4$ and $q>2$, satisfying the following properties :
\begin{itemize}
\item $3\leq r\leq q^2-1$;
\item there does not exist a plane $\pi$ such that every line of $\pi$ intersects $\pi\cap K$ in $r$ or $q^2+1$ points;
\end{itemize}
Then the set $K$ is the point set of a non-singular Hermitian variety $H(n,q^2)$.
\end{theorem}
The result below was obtained by Tallini in \cite{Tall1} and \cite{Tall2}.
\begin{theorem}\label{quadric}
In $PG(n,q)$ with $n\geq 4$ and $q>2$, let $K$ be a non-singular point set of type $(0,1,2,q+1)$.\\
If $\frac{q^{n+1}-1}{q-1}>|K|\geq \frac{q^n-1}{q-1}$ then one of the following holds:
\begin{itemize}
\item[(i)] $|K|=\frac{q^{n}-1}{q-1}$, $n$ is even and $K$ is the point set of a $Q(n,q)$.
\item[(ii)]$|K|=\frac{q^{n}-1}{q-1}+q^{\frac{n-1}{2}}$, $n$ is odd and $K$ is the point set of a $Q^+(n,q)$.
\item[(iii)]$|K|=\frac{q^{n}-1}{q-1}+1$, $q$ is even, and $K=\Pi_tK'\cup\{N\}$ with $\Pi_{t}$ some $\mathrm{PG}(t,q)\subset\mathrm{PG}(n,q)$ 
and with $K' $ (the point set of) a $Q(n-t-1,q)$ in some $(n-t-1)$-dimensional subspace of $\mathrm{PG}(n,q)$ skew from $\mathrm{PG}(t,q)$ 
(hence $n-t-1$ is even) or with $K'$ a $(q+1)$-arc in plane skew from $\mathrm{PG}(t,q)$ if $t=n-3$. In each case $N$ is the nucleus of $K'$.
\end{itemize}
\end{theorem}

These two theorems provide characterizations of certain polar spaces by their line intersections. It is a natural question to ask whether
polar spaces can also be characterized by their intersections with respect to other subspaces. A first result in this direction was given by the
following theorem of Ferri-Tallini \cite{FT}.
\begin{theorem}\label{Tallini}
If a set $K$ of points in $\mathrm{PG}(n,q),\:n\geq 4$, with $\left|K\right|\geq q^3+q^2+q+1$, is such that it has as intersection numbers with planes $1,\:a,\:b$, where $a\leq b$ and $b\geq 2q+1$, and as intersection numbers with solids $c,\:c+q,\:c+2q$, and such that solids intersecting $K$ in $c$ and $c+q$ points exist, then $K$ is the point set of a $Q(4,q)$.    
\end{theorem}
This theorem is a clear motivation to look at the following questions:
\begin{quote}
Is it possible to characterize finite classical polar spaces by their intersection numbers with respect to planes and solids, respectively by their intersections with
respect to hyperplanes and subspaces of codimension $2$? 
\end{quote}
(Note that these questions of course do not make any sense for the polar space $W_{2n+1}(q)$, as this polar space comprises all points of its ambient projective space.)
The first question was answered affirmitively in Schillewaert \cite{Schille}, and Schillewaert and Thas \cite{JSJAT}; the second question will be the subject of this paper. Though it is
possible to characterize certain polar spaces only by their line intersections, the existence (in abundance) of quasi-quadrics and quasi-Hermitian
varieties shows that it is not possible to characterize them merely by their intersections with respect to hyperplanes. Here we define a quasi-quadric, respectively quasi-Hermitian variety,
 to be a subset of the set of points of a projective space having the same intersection numbers with respect to hyperplanes as a non-singular quadric, respectively a non-singular Hermitian variety. In the parabolic case this is a slight deviation of the standard definition as given in \cite{ FDC}.
 The concept of a quasi-Hermitian variety in fact does not appear in the literature, but in unpublished work the second author showed that examples can be constructed with the same techniques as used to construct quasi-quadrics. For an overview on quasi-quadrics we refer to \cite{FDC}.

Hence the following Segre-type characterization of polar spaces, which is the main theorem of this paper, also provides necessary and sufficient conditions for a
quasi-quadric or a quasi-Hermitian variety to be a non-singular quadric or Hermitian variety.

\begin{theorem}\label{Main}
If a point set $K$ in $\mathrm{PG}(n,q)$, $n\geq4$, $q>2$, has the same intersection numbers with respect to hyperplanes and subspaces of codimension $2$
as a polar space $P\in\{H(n,q), Q^+(n,q), Q^-(n,q),Q(n,q)\}$, then $K$ is the point set of a non-singular polar space $P$.  
\end{theorem}

{\bf Remark}
For $n=3$ the conclusion of the above theorem remains true in the Hermitian and hyperbolic case. This is easily seen using Theorem \ref{BL}. In the elliptic case the conclusion only remains true if $q$ is odd, as every non-classical ovoid provides a counter example. 
If $n=4$ the conclusion of the theorem is true for all $q$. In the parabolic case this is a corollary of Theorem \ref{Tallini}, see Schillewaert \cite{Schille}, and in the Hermitian case this result was obtained in Schillewaert and Thas \cite{JSJAT}. 
\medskip

In the next section we will handle the four different types of polar spaces one by one. The basic idea is to study certain structures in the dual 
projective space. However the elliptic and parabolic case will turn out to be harder than the other two cases, and especially the parabolic case 
will be more complex and interesting (this is basically due to the fact that there are more intersections with respect to hyperplanes).

\section{Proof of the Theorem \ref{Main}}
From here on we always assume that we work in a projective space of dimension at least four and that $q>2$.
\subsection{Hermitian varieties}
Let us first recall the intersections of an $H(n,q^2)$ in $\mathrm{PG}(n,q^2)$ with hyperplanes and subspaces of codimension 2. 
A hyperplane intersects an $H(n,q^2)$ either in an $H(n-1,q^2)$ or in a cone $pH(n-2,q^2)$. A subspace of codimension 2 intersects an $H(n,q^2)$ either in an $H(n-2,q^2)$, in a cone $pH(n-3,q^2)$ or in a cone $LH(n-4,q^2)$. Hence the intersection numbers with hyperplanes are 

\[H_1={\frac { \left( {q}^{n}-(-1)^{n} \right)  \left( {q}^{n-1}+(-1)^{n} \right) }{{q}^{2
}-1}},\:H_2=1+{\frac {{q}^{2} \left( {q}^{n-1}+(-1)^{n} \right)  \left( {q}^{n-2}-(-1)^{n}
 \right) }{{q}^{2}-1}}.\]

and the intersection numbers with subspaces of codimension 2 are 

\[C_1={\frac { \left( {q}^{n-1}+(-1)^n \right)  \left( {q}^{n-2}-(-1)^n \right) }{{q}^{2}-1}},\:C_2=1+{\frac {{q}^{2} \left( {q}^{n-2}-(-1)^n \right)  \left( {q}^{n-3}+(-1)^n \right) }{{q}^{2}-1}},\]
\[C_3=1+{q}^{2}+{\frac {{q}^{4} \left( {q}^{n-3}+(-1)^n \right)  \left( {q}^{n-4} -(-1)^n \right) }{{q}^{2}-1}}.\]

From now on, let $K$ be a point set in $\mathrm{PG}(n,q^2)$ having the above intersection numbers with respect to hyperplanes and subspaces of 
codimension $2$. We want to prove that $K$ is the point set of an $H(n,q^2)$. We will call subspaces intersecting $K$ in a given
number $m$ of points, {\em subspaces of type $m$}. For obvious reasons a hyperplane intersecting
$K$ in $H_{2}$ points will also be called a {\em tangent hyperplane}.

\begin{lemma}
The set $K$ contains $|H(n,q^2)|$ points. There are $|H(n,q^2)|$ tangent hyperplanes. 
\end{lemma}
\begin{proof}
We count in two ways the pairs $(p,\alpha)$ where $p$ is a point of $K$ and $\alpha$ a hyperplane such that $p\in\alpha$, respectively the triples 
$(p,r,\alpha)$ where $p\neq r$ are points of $K$ and $\alpha$ a hyperplane such that $p,r\in\alpha$. Denote by $h_{1}$ the number of hyperplanes 
intersecting $K$ in $H_1$ points and by $x$ the size of $K$. We obtain:

\begin{equation}\label{Herm1}
x\frac{q^{2n}-1}{q^2-1}=h_{1}H_{1}+(\frac{q^{2n+2}-1}{q^2-1}-h_{1})H_{2},
\end{equation}

and

\begin{equation}\label{Herm2}
x(x-1)\frac{q^{2n-2}-1}{q^2-1}=h_{1}H_{1}(H_{1}-1)+(\frac{q^{2n+2}-1}{q^2-1}-h_{1})H_{2}(H_{2}-1).
\end{equation}

Solving the first equation for $h_{1}$ and substituting this value in the second equation yields a quadratic equation in $x$. The solutions are 
$x_1=|H(n,q^2)|$ and 
$x_2$, which is a tedious expression in $q$, however easily computed with any computer algebra package.
 
We show the latter is impossible. So suppose by way of contradiction that the set $K$ contains $x_2$ points. Consider a subspace $\Pi$ of 
codimension $2$ intersecting the set $K$ in $C_i$ points. Denote by $k_{i}$ the number of non-tangent hyperplanes containing $\Pi$. We obtain 
the following equation.
\[k_{i}(H_1-C_i)+(q^2+1-k_{i})(H_2-C_i)+C_i=x_2.\]
Solving this equation in $k_{i}$ for $i=1,2,3$ yields respectively 
\[k_{1}={\frac {{q}^{n}-(-1)^n({q}^{2}-q+1)}{{q}^{n-1}-(1)^n}} ; \mathrm{\ } k_{2}={\frac {2\,{q}^{n}-(-1)^n({q}^{2}+1)}{{q}^{n-1}-(-1)^n}};\]
\[ \mathrm{\ }k_{3}={-\frac {{q}^{n+1}-2\,{q}^{n}+(-1)^n}{{q}^{n-1}-(-1)^n}}.\] 
These are not natural numbers if $n>2$, proving $x_{2}$ cannot occur. 

The second assertion follows by substituting $x=\left| H(n,q^2)\right|$ in Equation \ref{Herm1}.\qed
\end{proof}

{\bf Remark.}
Notice that for $n=2$, we would obtain natural numbers and in that case we have $x_{2}=q^2+q+1$, that is, exactly the number of points of a Baer subplane, which was to be expected.

\begin{lemma}\label{linesizeHerm}
Through a space of codimension 2 of type $C_1,C_2,C_3$ there pass respectively $T_1=q+1,\:T_2=1,\:T_3=q^2+1$ tangent hyperplanes. 
\end{lemma}
\begin{proof}
Let $\Pi$ be a codimension $2$-space intersecting $K$ in $C_i$ points and let $T_i$ denote the number of tangent hyperplanes containing $\Pi$. We obtain the following equation:
$$C_i+T_i(H_2-C_i)+(q^2+1-T_i)(H_1-C_i)=|K|.$$
Solving the equation in $T_i$ for $i=1,2,3$ yields the result. 
\qed
\end{proof}

\begin{lemma}\label{non-singular}
Each tangent hyperplane contains $A_i$ subspaces of codimension $2$ intersecting $K$ in $C_{i}$ points, where 
\[A_{1}=q^{2n-2}, \mathrm{\ } A_{2}={\frac {q^{n-2} \left( q^{n-1}+(-1)^n \right) }{q+1}},\]
\[A_{3}={\frac {(q^{n-1}+(-1)^n)(q^{n-2}-(-1)^n)}{q^2-1}}.\].
\end{lemma}
\begin{proof}
Consider any tangent hyperplane $\Pi$. Denote by $A_{i}$ the number of codimension $2$ subspaces of type $C_{i}$ contained in $\Pi$. 
Then
$$\sum_i A_{i}=\frac{q^{2n}-1}{q^2-1}.$$
We count in two ways the  pairs $(p,\Delta)$, $p\in\Delta\subset\Pi$, with $p$ a point of $K$ and $\Delta$ a subspace of codimension $2$. We obtain
\[\sum_i A_{i}C_{i}=H_{2}\frac{q^{2n-2}-1}{q^2-1}.\]
Next we count in two ways the triples $(p,r,\Delta)$, with $p,r\in\Delta\subset\Pi$, with $p\neq r$ points of $K$ and 
$\Delta$ a subspace of codimension $2$. We obtain
\[\sum_i A_iC_i(C_i-1)=H_{2}(H_{2}-1)\frac{q^{2n-4}-1}{q^2-1}.\]
The obtained system of three linear equations in $A_{1},\:A_{2}$ and $A_{3}$ can easily be solved, yielding the desired result.
\qed
\end{proof}

\begin{lemma}\label{raakhypperpunt}
Each point of $K$ is contained in $H_2$ tangent hyperplanes, while each point not in $K$ is contained in $H_1$ tangent hyperplanes. 
\end{lemma}
\begin{proof}
We need to show that each point of $K$ is contained in $H_2$ tangent hyperplanes. Set $K=\{p_1,\cdots,p_{\left|H(n,q^2)\right|}\}$. Let $a_i$ denote the number of tangent hyperplanes containing the point $p_i$. Counting pairs $(p,\tau)$, $p\in K$, $p\in\tau$, $\tau$ a tangent hyperplane we obtain:
\begin{equation}\label{somHerm}
\sum_i a_i= |H(n,q^2)|H_2.
\end{equation}
Next we count triples $(p,\tau_1,\tau_2)$,  $p\in K$, $p\in\tau_i$, $\tau_i$ a tangent hyperplane, $i=1,2$, $\tau_1 \neq \tau_2$.
This yields the following equation.
\begin{equation}\label{somkwadHerm}
\sum_i a_i(a_i-1)= \left|H(n,q^2)\right|\left(\sum_{j} A_j(T_j-1)C_j \right),
\end{equation}
where as before $T_i$ denotes the number of tangent hyperplanes containing a fixed codimension $2$-space $\Pi$ of type $C_i$ and $A_{i}$ is the number of codimension $2$ subspaces of type $C_{i}$ in a tangent hyperplane.

From Equations (\ref{somHerm}) and (\ref{somkwadHerm}) we can compute 
\[\left|H(n,q^2)\right|\sum_i a_i^2 - (\sum_i a_i)^2=0,\] 
from which we deduce, using the variance trick, that $a_i$ is a constant equal to $H_2$.

The second assertion is proved in a similar way.
\qed
\end{proof}

Denote by $\mathcal{H}$ the set of tangent hyperplanes of $K$. Let $\delta:\mathrm{PG}(n,q)\rightarrow\mathrm{PG}(n,q)^D$ be any fixed chosen
duality of $\mathrm{PG}(n,q)$.  

\begin{lemma}\label{intersection-herm}
The point set $K':=\mathcal{H}^\delta$ is a $(1,q+1,q^2+1)$-set in $\mathrm{PG}(n,q)^D$.
\end{lemma}
\begin{proof}
As by Lemma \ref{linesizeHerm} a codimension $2$ subspace is contained in either $1$, $q+1$ or $q^2+1$ tangent hyperplanes, applying $\delta$ immediately yields the result.
\qed
\end{proof}

\begin{lemma}\label{char-herm}
The set $K'$ is the point set of a non-singular Hermitian variety $H(n,q^2)$ in $\mathrm{PG}(n,q)^D$.
\end{lemma}
\begin{proof}
We will check the conditions of Theorem \ref{Hermitian}. Since every tangent hyperplane contains subspaces of type $C_{1}$ (by Lemma \ref{non-singular}), we see that every point
of $K'$ is contained in lines intersecting $K'$ in exactly $q+1$ points. Hence $K'$ is non-singular. As $q>2$ also the first condition is satisfied. Now assume there would be a plane $\pi$ intersecting $K'$ in a point set such that every line of $\pi$
would intersect $K'\cap \pi$ in $q+1$ or $q^2+1$ points. Then $K'\cap\pi$ would be the complement of a maximal $(q^2-q)$-arc in $\pi$, implying
that $q^2-q$ divides $q^2$, a contradiction since $q>2$. Consequently also the second condition of theorem \ref{Hermitian} is satisfied and $K'$
is the point set of a non-singular Hermitian variety in $\mathrm{PG}(n,q^2)$. 
\qed
\end{proof}

\begin{theorem}\label{theoremHermitian}
The set $K$ is the point set of a non-singular Hermitian variety $H(n,q^2)$.
\end{theorem}
\begin{proof}
Clearly $K'$ is a point set satisfying the same conditions on intersections with hyperplanes and subspaces of codimension $2$ as $K$ in $\mathrm{PG}(n,q^2)$. Since the tangent hyperplanes at $K'$ are exactly those hyperplanes containing $H_2$ points of $K'$, we find that if we apply the duality $\delta^{-1}$ to $\mathrm{PG}(n,q)^D$, the tangent hyperplanes of $K'$ are mapped bijectively to the points of $K$ (see Lemma \ref{raakhypperpunt}). Hence, we can now apply Lemma \ref{char-herm} with $K'$ in $\mathrm{PG}(n,q^2)^D$ replaced by $K$ in $\mathrm{PG}(n,q^2)$. We conclude that $K$ is the point set of an $H(n,q^2)$.
\qed
\end{proof}

\subsection{Hyperbolic quadrics}
First of all let us recall what the intersections of a $Q^+(2n+1,q)$ with hyperplanes and spaces of codimension 2 look like. A hyperplane can intersect a $Q^+(2n+1,q)$ in a $Q(2n,q)$ or in a cone $pQ^+(2n-1,q)$. 
A space of codimension 2 can intersect a $Q^+(2n+1,q)$ either in a $Q^-(2n-1,q)$, in a cone $pQ(2n-2,q)$,  in a $Q^+(2n-1,q)$ or in a cone $LQ^+(2n-3,q)$.

So the intersection numbers with hyperplanes are
$$H_1=\frac{q^{2n}-1}{q-1},\:H_2=1+q\frac{(q^n-1)(q^{n-1}+1)}{q-1}.$$

The intersection numbers with spaces of codimension 2 are
$$C_1=\frac{(q^n+1)(q^{n-1}-1)}{q-1},\:C_2=1+{\frac {q \left( {q}^{2\,n-2}-1 \right) }{q-1}},\:$$
$$\:C_3={\frac { \left( {q}^{n}-1 \right)  \left( {q}^{n-1}+1 \right) }{q-1}},C_4=1+q+{\frac {{q}^{2} \left( {q}^{n-1}-1 \right)  \left( {q}^{n-2}+1
 \right) }{q-1}}.$$

From now on let $K$ be a set of points in $\mathrm{PG}(2n+1,q)$ having the same intersection numbers with hyperplanes and codimension $2$-spaces as a $Q^+(2n+1,q)$. We want to prove that $K$ is the point set of a $Q^+(2n+1,q)$. We will call subspaces intersecting $K$ in a given
number $m$ of points, {\em subspaces of type $m$}. For obvious reasons a hyperplane intersecting
$K$ in $H_{2}$ points will also be called a {\em tangent hyperplane}.

\begin{lemma}\label{puntenHyp}
The set $K$ has size $|Q^+(2n+1,q)|$. Furthermore, there are $|Q^+(2n+1,q)|$ tangent hyperplanes. 
\end{lemma}
\begin{proof}
We count in two ways the pairs $(p,\alpha)$ where $p$ is a point of $K$ and $\alpha$ a hyperplane such that $p\in\alpha$, respectively the triples $(p,r,\alpha)$ where $p\neq r$ are points of $K$ and $\alpha$ a hyperplane such that $p,r\in\alpha$. Call $h_1$ the number of hyperplanes intersecting $K$ in $H_1$ points and $x$ the size of $K$. This yields the following equations
\begin{equation}\label{Hyp1}
h_1H_1+(\frac{q^{2n+2}-1}{q-1}-h_1)H_2=x\frac{q^{2n+1}-1}{q-1},
\end{equation}
\begin{equation}\label{Hyp2}
h_1H_1(H_1-1)+(\frac{q^{2n+2}-1}{q-1}-h_1)H_2(H_2-1)=x(x-1)\frac{q^{2n}-1}{q-1}.
\end{equation}
Solving $h_1$ in terms of $x$ from the first equation and substituting this in the second equation yields a quadratic equation in $x$. The solutions are $x_1=|Q^+(2n+1,q)|$ and 
$x_2$, which is, as in the Hermitian case, an easily computed but tedious expression in $q$.

We show that the latter cannot occur. So suppose the set $K$ contains $x_2$ points. Consider a space $\Pi$ of codimension 2 intersecting the set $K$ in $C_i$ points. Denote by $k_i$ the number of non-tangent hyperplanes containing $\Pi$. Then we obtain 
$$k_i(H_1-C_i)+(q+1-k_i)(H_2-C_i)+C_i=x_2.$$
Solving this equation in $k_{i}$ for $i=1,2,3,4$ yields
\[k_1=3+\frac{q-1}{q^n+1}, \mathrm{\ } k_2=q\frac{q^{n-1}+1}{q^n+1},\] \[k_3=\frac{2q^n+q+1}{q^n+1}, \mathrm{\ } k_4=-{\frac {{q}^{n+1}-2\,{q}^{n}-1}{{q}^{n}+1}}.\]
These are not natural numbers, the desired contradiction.

The second assertion follows by substituting $x=\left|Q^+(2n+1,q)\right|$ in Equation (\ref{Hyp1}).
\qed
\end{proof}

\begin{lemma}\label{linesize}
Through a space of codimension 2 of type $C_1,C_2,C_3,C_4$ there pass respectively $T_1=0,\:T_2=1,\:T_3=2$ and $T_4=q+1$ tangent hyperplanes. 
\end{lemma}
\begin{proof}
Let $\Pi$ be a codimension $2$-space intersecting $K$ in $C_i$ points and let $T_i$ denote the number of tangent hyperplanes containing $\Pi$. We obtain the following equation.
$$C_i+T_i(H_2-C_i)+(q+1-T_i)(H_1-C_i)=|K|.$$
Solving the equation in $T_i$ for $i=1,2,3$ yields the result.
\qed
\end{proof}

\begin{lemma}\label{codim2hyp}
Each tangent hyperplane contains $A_i$ subspaces of codimension $2$ intersecting $K$ in $C_{i}$ points, where 
\[A_{1}=0, \mathrm{\ } A_{2}=q^{n-1}(q^n-1), \mathrm{\ } A_{3}=q^{2n},\]
\[A_4=\frac{(q^n-1)(q^{n-1}+1)}{q-1}.\]
\end{lemma}
\begin{proof}
Consider any tangent hyperplane $\Pi$. By the previous lemma $A_1=0$. 
Hence  
\[\sum_{i=2}^{4}A_{i}=\frac{q^{2n+1}-1}{q-1}.\]
We count in two ways the  pairs $(p,\Delta)$, $p\in\Delta\subset\Pi$, with $p$ a point of $K$ and $\Delta$ a subspace of codimension $2$. We obtain
\[\sum_{i=2}^{4} A_iC_i=H_{2}\frac{q^{2n}-1}{q-1}.\]
Next we count in two ways the triples $(p,r,\Delta)$, with $p,r\in\Delta\subset\Pi$, with $p\neq r$ points of $K$ and 
$\Delta$ a subspace of codimension $2$. We obtain
\[\sum_{i=2}^{4} A_iC_i(C_i-1)=H_{2}(H_{2}-1)\frac{q^{2n-1}-1}{q-1}.\]
The obtained system of three linear equations in $A_{2},A_{3}$ and $A_{4}$ can easily be solved, yielding the desired result.
\qed
\end{proof}

\begin{lemma} \label{correspondence}
Each point of $K$ is contained in $H_2$ tangent hyperplanes, while each point not in $K$ is contained in $H_1$ tangent hyperplanes.
\end{lemma}
\begin{proof}
We need to show that each point of $K$ is contained in $H_2$ tangent hyperplanes. Set $K=\{p_1,\cdots,p_{\left|Q^+(2n+1,q)\right|}\}$. Let $a_i$ denote the number of tangent hyperplanes containing the point $p_i$. Counting pairs $(p,\tau)$, $p\in K$, $p\in\tau$, $\tau$ a tangent hyperplane we obtain:
\begin{equation}\label{somHyp}
\sum_i a_i= |Q^+(2n+1,q)|H_2.
\end{equation}
Next we count triples $(p,\tau_1,\tau_2)$,  $p\in K$, $p\in\tau_i$, $\tau_i$ a tangent hyperplane, $i=1,2$.
As before $T_i$ denotes the number of tangent hyperplanes containing a fixed codimension $2$-space $\Pi$ of type $C_i$. 
We obtain the following equation.
\begin{equation}\label{somkwadHyp}
\sum_i a_i(a_i-1)= \left|Q^+(2n+1,q)\right|\left(\sum_{j} A_j(T_j-1)C_j \right).
\end{equation}

From Equations (\ref{somHyp}) and (\ref{somkwadHyp}) we can compute 
\[\left|Q^+(2n+1,q)\right|\sum_i a_i^2 - (\sum_i a_i)^2=0,\] 
from which we deduce that $a_i$ is a constant equal to $H_2$.

The second assertion is proved in a similar way.
\qed
\end{proof}

Denote by $\mathcal{H}$ the set of tangent hyperplanes of $K$.\\ Let $\delta:\mathrm{PG}(2n+1,q)\rightarrow \mathrm{PG}(2n+1,q)^D$ be any fixed chosen duality of $\mathrm{PG}(2n+1,q)$.

\begin{lemma}\label{char-hyp}
The set $K':=\mathcal{H}^\delta$ is the point set of a non-singular hyperbolic quadric $Q^+(2n+1,q)$ in $\mathrm{PG}(2n+1,q)^D$.
\end{lemma}
\begin{proof}
This is an immediate consequence of Lemma \ref{linesize}, Lemma \ref{codim2hyp} and Theorem \ref{quadric}.\qed
\end{proof}
\begin{theorem}\label{hyperbolic}
The set $K$ is the point set of a non-singular hyperbolic quadric $Q^+(2n+1,q)$ in $\mathrm{PG}(2n+1,q)$.
\end{theorem}
\begin{proof}
Clearly $K'$ is a point set satisfying the same conditions on intersections with hyperplanes and subspaces of codimension $2$ as $K$ in $\mathrm{PG}(2n+1,q)$. Since the tangent hyperplanes at $K'$ are exactly those hyperplanes containing $H_2$ points of $K'$, we find that if we apply the duality $\delta^{-1}$ to $\mathrm{PG}(2n+1,q)^D$, the tangent hyperplanes of $K'$ are mapped bijectively to the points of $K$ (see Lemma \ref{correspondence}). Hence we can now apply Lemma \ref{char-hyp} with $K'$ in $\mathrm{PG}(2n+1,q)^D$ replaced by $K$ in $\mathrm{PG}(2n+1,q)$. 
We conclude that $K$ is the point set of a $Q^+(2n+1,q)$.\qed
\end{proof}

\subsection{Elliptic quadrics}

First of all let us recall what the intersections of a $Q^-(2n+1,q)$ with hyperplanes and spaces of codimension 2 look like. A hyperplane of $\mathrm{PG}(2n+1,q)$ can intersect a $Q^-(2n+1,q)$ either in a $Q(2n,q)$ or in a cone $pQ^-(2n-1,q)$. 

A space of codimension 2 can intersect a $Q^-(2n+1,q)$ either in a \\ $Q^+(2n-1,q)$, a cone $pQ(2n-2,q)$, a $Q^-(2n-1,q)$ or a cone $LQ^-(2n-3,q)$.

So the intersection numbers with hyperplanes are
$$H_1=\frac{q^{2n}-1}{q-1},H_2=1+q\frac{(q^n+1)(q^{n-1}-1)}{q-1}.$$

The intersection numbers with spaces of codimension 2 are 
$$\:C_1={\frac { \left( {q}^{n}-1 \right)  \left( {q}^{n-1}+1 \right) }{q-1}},\:C_2=1+{q\frac {\left( {q}^{2\,n-2}-1 \right) }{q-1}},$$
$$C_3={\frac { \left( {q}^{n}+1 \right)  \left( {q}^{n-1}-1 \right) }{q-1}},\:C_4=1+q+{q^2\frac { \left( {q}^{n-1}+1 \right)  \left( {q}^{n-2}-1
 \right) }{q-1}}.$$ 

From now on let $K$ be a set of points in $\mathrm{PG}(2n+1,q)$, $n\geq2$ having the same intersection numbers with hyperplanes and codimension $2$-spaces as a $Q^-(2n+1,q)$. We want to prove that $K$ is the point set of a non-singular elliptic quadric. We will call subspaces intersecting $K$ in a given
number $m$ of points, {\em subspaces of type $m$}. For obvious reasons a hyperplane intersecting
$K$ in $H_{2}$ points will also be called a {\em tangent hyperplane}.

\begin{lemma}
The set $K$ has size $|Q^-(2n+1,q)|$. Furthermore, there are $|Q^-(2n+1,q)|$ tangent hyperplanes. 
\end{lemma}
\begin{proof}
This is proved in a similar way as Lemma \ref{puntenHyp}.\qed
\end{proof}
\begin{lemma}\label{linesize-ell}
Through a space of codimension 2 of type $C_1,C_2,C_3,C_4$ there pass respectively $T_1=0,\:T_2=1,\:T_3=2$ and $T_4=q+1$ tangent hyperplanes.
\end{lemma}
\begin{proof}
This is completely analogous to the proof of Lemma \ref{linesize}.
\qed
\end{proof}

\begin{lemma}
Each tangent hyperplane contains $A_i$ subspaces of codimension $2$ intersecting $K$ in $C_{i}$ points, where 
\[A_{1}=0, \mathrm{\ } A_{2}=q^{n-1}(q^n+1), \mathrm{\ } A_{3}=q^{2n},\]
\[A_4=\frac{(q^n+1)(q^{n-1}-1)}{q-1}.\]

\end{lemma}
\begin{proof}
The proof is similar to the proof of Lemma \ref{codim2hyp}.\qed
\end{proof}

\begin{lemma}\label{correspondence-ell} 
Each point of $K$ is contained in $H_2$ tangent hyperplanes, while each point not in $K$ is contained in $H_1$ tangent hyperplanes.
\end{lemma}
\begin{proof}
This is proved as in Lemma \ref{correspondence}.\qed
\end{proof}

Denote by $\mathcal{H}$ the set of tangent hyperplanes of $K$.\\ Let $\delta:\mathrm{PG}(2n+1,q)\rightarrow \mathrm{PG}(2n+1,q)^D$ be any fixed chosen duality of $\mathrm{PG}(2n+1,q)$.

By Lemma \ref{linesize-ell} the set $K':=\mathcal{H}^\delta$ is a $(0,1,2,q+1)$-set in $\mathrm{PG}(2n+1,q)^D$. We want to show that $K'$ is the point set of an elliptic quadric. 
Notice however that $K'$ does not satisfy the conditions of Theorem \ref{quadric}. By Lemma \ref{correspondence-ell} the intersection numbers of  $K'$ with respect to hyperplanes are $H_1$ and $H_2$.

We define a point-line geometry $S$, with point set $K'$ and line set those lines of  $\mathrm{PG}(2n+1,q)^D$ intersecting $K'$ in $q+1$ points (the incidence is the natural one). 
\begin{theorem}
The geometry $S$ is a Shult space such that no point of $S$ is collinear with all other points of $S$.
\end{theorem}
\begin{proof}
Consider a point $p$ of $S$ and a line $L$ of $S$, such that $p$ and $L$ are not incident. Consider the plane $\alpha$ generated by $p$ and $L$.

If this plane contains another point $r$ of $S$, then the fact that  $K'$ is a $(0,1,2,q+1)$-set implies that $\alpha$ intersects $S$ either in two intersecting lines or is fully contained in $S$ (notice that $q>2$ is necessary here). In both cases the ``1-or-all axiom'' holds.

Next suppose that $\alpha$ intersects $S$ only in $p$ and $L$.
Assume that $\alpha$ would not be contained in a hyperplane of type $H_1$. We count pairs $(u,H)$, with $u$ a point of  $S$ not in $\alpha$ and $H$ a hyperplane containing $u$ and $\alpha$. We obtain
\[\left(\left|Q^-(2n+1,q)\right|-q-2\right)\frac{q^{2n-2}-1}{q-1}=\frac{q^{2n-1}-1}{q-1}(H_2-q-2),\]
a contradiction.
Hence $\alpha$ is contained in at least one hyperplane $\Pi$ of type $H_1$. By Theorem \ref{quadric} the  $(0,1,2,q+1)$-set $K'\cap\Pi$ is the point set of a non-singular parabolic quadric. We conclude that in $S$ the point $p$ is collinear with $1$ or all points of $L$. 
Since each point of $S$ is contained in at least one hyperplane of type $H_1$ no point of $S$ is collinear with all other points of $S$.
We still have to proof that through every point there passes a constant number of lines in order to conclude that $S$ is a Shult space (using our definition from the introduction of Shult-space which is slightly different from the original one). 
Consider first two non-collinear points $p$ and $r$. Then for every line containing $p$ there is exactly one line through $r$. Hence through $p$ and $r$ there pass the same number of lines. Next consider two collinear points $p$ and $r$. 
Considering all planes containing the line $pr$, it is clear that one can find a point $s$ which is both non-collinear with $p$ and $r$. By the  above,  there pass equally many lines through $p$ and $r$.
\qed
\end{proof}

\begin{corollary}\label{char-ell}
The point set $K'$ forms the point set of a non-singular elliptic quadric $Q^-(2n+1,q)$.
\end{corollary}
\begin{proof}
This is an immediate consequence of $\left|K'\right|=\left|Q^-(2n+1,q)\right|$, the previous lemma and Theorem \ref{BL}.\qed
\end{proof}

\begin{theorem}\label{elliptic}
The set $K$ is the point set of a non-singular elliptic quadric $Q^-(2n+1,q)$ in $\mathrm{PG}(2n+1,q)$.
\end{theorem}
\begin{proof}
Clearly $K'$ is a point set satisfying the same conditions on intersections with hyperplanes and subspaces of codimension $2$ as $K$ in $\mathrm{PG}(2n+1,q)$. Since the tangent hyperplanes at $K'$ are exactly those hyperplanes containing $H_2$ points of $K'$, we find that if we apply the duality $\delta^{-1}$ to $\mathrm{PG}(2n+1,q)^D$, the tangent hyperplanes of $K'$ are mapped bijectively to the points of $K$ (see Lemma \ref{correspondence-ell}). Hence we can now apply Corollary \ref{char-ell} with $K'$ in $\mathrm{PG}(2n+1,q)^D$ replaced by $K$ in $\mathrm{PG}(2n+1,q)$. We conclude that $K$ is the point set of a $Q^-(2n+1,q)$.\qed
\end{proof}

\subsection{Parabolic quadrics}

First of all let us recall what the intersections of a $Q(2n,q)$ with hyperplanes and spaces of codimension 2 look like. A hyperplane can intersect a $Q(2n,q)$ either in a $Q^-(2n-1,q)$, a $Q^+(2n-1,q)$ or in a cone $pQ(2n-2,q)$. 

A space of codimension 2 can intersect a parabolic quadric $Q(2n,q)$ either in a $Q(2n-2,q)$, a cone $pQ^+(2n-3,q)$, a cone $pQ^-(2n-3,q)$ or a cone $LQ(2n-4,q)$ (notice that this cone contains the same number of points as a $Q(2n-2,q)$).
So the intersection numbers with hyperplanes are
$$H_1=\frac{(q^n-1)(q^{n-1}+1)}{q-1},H_2=\frac{(q^n+1)(q^{n-1}-1)}{q-1},$$
$$H_3=1+q\frac{q^{2n-2}-1}{q-1}.$$

The intersection numbers with spaces of codimension 2 are 
$$\:C_1=\frac{q^{2n-2}-1}{q-1},\:C_2=1+q\frac{(q^{n-1}-1)(q^{n-2}+1)}{q-1},$$
$$C_3=1+q\frac{(q^{n-1}+1)(q^{n-2}-1)}{q-1}.$$ 

From now on let $K$ be a set of points in $\mathrm{PG}(2n,q)$ having the same intersection numbers with hyperplanes and codimension $2$-spaces as a $Q(2n,q)$. We want to prove that $K$ is the point set of a $Q(2n,q)$. We will call subspaces intersecting $K$ in a given
number $m$ of points, {\em subspaces of type $m$}. For obvious reasons a hyperplane intersecting
$K$ in $H_{2}$ points will also be called a {\em tangent hyperplane}.

\begin{lemma}
The set $K$ contains $\left|Q(2n,q)\right|$ points.
\end{lemma}
\begin{proof}
Let $h_i$, respectively $c_i$, denote the number of hyperplanes of type $H_i$, respectively the number of codimension $2$ spaces of type $C_i$. By counting pairs and triples as in Lemma \ref{puntenHyp}, but now with respect to hyperplanes as well as with respect to codimension
$2$ spaces, we obtain

\begin{eqnarray}
\sum_i h_i=\frac{q^{2n+1}-1}{q-1},\label{one}
\\
\sum_i h_iH_i = \frac{q^{2n}-1}{q-1}\left|K\right|\label{two},
\\
\sum_i h_iH_i(H_i-1)= \frac{q^{2n-1}-1}{q-1}\left|K\right|(\left|K\right|-1)\label{three},
\\
\sum_i c_i=\frac{(q^{2n+1}-1)(q^{2n}-1)}{(q^2-1)(q-1)}\label{four},
\\
\sum_i c_iC_i=\left|K\right|\frac{(q^{2n}-1)(q^{2n-1}-1)}{(q^2-1)(q-1)}\label{five}, 
\\
\sum_i c_iC_i(C_i-1)=\left|K\right|(\left|K\right|-1)\frac{(q^{2n-1}-1)(q^{2n-2}-1)}{(q^2-1)(q-1)}\label{six}. 
\end{eqnarray}

Now consider a hyperplane $\Pi$ of type $H_j$ and denote by $m_i^j$ the number of codimension $2$ spaces of type $C_i$ it contains. By counting pairs $(p,\Delta)$ and triples $(p,r,\Delta)$, with $p\neq r$ points of $K\cap\Pi$, $\Delta\subset\Pi$ 
a codimension $2$ space, and $p,r\in\Delta$, we obtain
\[\sum_i m_i^j=\frac{q^{2n}-1}{q-1}\]
\[\sum_i m_i^jC_i=H_j\frac{q^{2n-1}-1}{q-1}\]
\[\sum_i m_i^jC_i(C_i-1)=H_j(H_j-1)\frac{q^{2n-2}-1}{q-1}\] 

From these equations all values $m_i^j$ can easily be determined. As $m_2^2=0$ we see that the number of hyperplanes of type $H_1$ through a codimension $2$ space of type $C_{2}$ is a constant $f(K)$ depending only on the size of $K$, with $f(K)$ linear in $\left|K\right|$.

By counting pairs $(H,\Delta)$, with $H$ a hyperplane of type $H_{1}$, $\Delta\subset H$ a codimension $2$ space of type $C_{2}$, we obtain
\begin{equation}\label{seven}
h_1m_2^1=c_2f(K).
\end{equation}

So $h_1=h_1(c_2,K)$, with $h_{1}(c_{2},K)$ linear in $\left|K\right|$. From Equations (\ref{four}), (\ref{five}) and (\ref{six}) we can obtain an expression for $c_2$ depending only on $\left|K\right|$, say $c_2=c_2(K)$, with $c_{2}(K)$ quadratic in $\left|K\right|$. 
From Equations (\ref{one}), (\ref{two}) and (\ref{three}) we can now obtain a quadratic equation in $\left|K\right|$ of the form $s\left|K\right|^2+t\left|K\right|+h_1=0$. By substitution we obtain 
\begin{equation}\label{K}
s\left|K\right|^2+t\left|K\right|+h_1(c_2(K),K)=0,
\end{equation}
which turns out to be a cubic equation in $\left|K\right|$.  
One simply checks that $\left|K\right|=\frac{q^{2n}-1}{q-1}$ is a solution of Equation (\ref{K}). 

We want to exclude the other roots of Equation (\ref{K}) as possible sizes for the set $K$.

Though Maple is not able to directly calculate the other roots for general $n$ and $q$ it is not too hard to determine the product and sum of the roots of Equation (\ref{K}) with it. 

One obtains that the three roots have product
$$
{\frac {{q}^{4n-2}+q^{2n+1}-3q^{2n}+q^{2n-1}-q^{2n-2}+1}{\left( q-1 \right) ^{3}
 \left( {q}^{2\,n-1}-1 \right) }},
$$

and sum
$$
3\,{\frac { \left( {q}^{n}+1 \right)  \left( {q}^{n}-1 \right) }{q-1}}.
$$
From the above expressions one can deduce that the roots different from $\frac{q^{2n}-1}{q-1}$ are not real numbers, a contradiction.
\end{proof}

\begin{lemma}\label{aantal}
There are exactly $\frac{q^{2n}-1}{q-1}$ tangent hyperplanes. Furthermore,\\ every codimension $2$ space of type $C_2$ or $C_3$ is contained in exactly one tangent hyperplane.
\end{lemma}
\begin{proof}
The first assertion follows from Equations (\ref{one}), (\ref{two}) and (\ref{three}) once we know $\left|K\right|=\frac{q^{2n}-1}{q-1}$.
To prove the second assertion, we notice that $m_3^1=m_2^2=0$. Hence, if $T_i$ denotes the number of tangent hyperplanes containing a given codimension $2$ space of type $C_i$, $i=2,3$, then
\[T_i(H_3-C_i)+(q+1-T_i)(H_{i-1}-C_i)+C_i=\left|K\right|.\]
We obtain $T_2=T_3=1$.
\qed
\end{proof}

\begin{lemma}\label{equal}
For every codimension $2$-space of type $C_1$ the number of hyperplanes of type $H_1$ containing it is equal to the number of hyperplanes of type $H_2$ in which it is contained.
\end{lemma}
\begin{proof}
One notices that $\left|K\right|=(q+1)(H_3-C_1)+C_1$ and that $(H_1+H_2)/2=H_3$. The lemma follows.
\qed
\end{proof}
\begin{lemma}
Let $\gamma$ be  a codimension $3$ space contained in a hyperplane $H$ of type $H_1$. Suppose that $\gamma$ is contained in $N_H$ codimension $2$ spaces $\alpha$ of type $C_2$, such that $\gamma\subset\alpha\subset H$. 
Then $\left|\gamma\cap K\right|=q^{n-2}N_{H}+\frac{(q^{n-1}+1)(q^{n-2}-1)}{q-1}$. Furthermore if $\gamma$ is also contained in a hyperplane $E$ of type $H_2$, then $N_H\leq2$.
\end{lemma}
\begin{proof}
Let $X$ denote the number of points of $K$ contained in $\gamma$. As $m_3^1=0$ we have
\[(q+1-N_H)(C_1-X)+N_H(C_2-X)+X=H_1.\]
It follows that 
\[X=N_Hq^{n-2}+\frac{(q^{n-1}+1)(q^{n-2}-1)}{q-1}.\]
Next let $E$ be a hyperplane of type $H_2$ containing $\gamma$ and let $N_E$ be the number of codimension $2$ spaces $\beta$ of type $C_3$ such that $\gamma\subset\beta\subset E$. Since $m_2^2=0$ we obtain that
\[(q+1-N_E)(C_1-X)+N_E(C_3-X)+X=H_2.\]
Substitution of the higher obtained expression for $X$ in terms of $N_H$ yields 
\[N_E=2-N_H.\]
As $N_E\geq0$ the lemma follows.  
\qed
\end{proof}

\begin{lemma}
A codimension $3$ space that is contained in a hyperplane of type $H_1$ contains $Nq^{n-2}+\frac{(q^{n-1}+1)(q^{n-2}-1)}{q-1}$ points of $K$, with $N\in\{0,1,2,q+1\}$.
\end{lemma}
\begin{proof}
Let $\gamma$ be any codimension $3$ space contained in a hyperplane of type $H_{1}$, and set $X=\left|\gamma\cap K\right|$ . If $\gamma$ is contained in hyperplanes of type $H_1$ as well as of type $H_2$ there is nothing to proof because of the previous lemma.

If $\gamma$ is contained in  no hyperplane of type $H_2$, then $\gamma$ cannot be contained in a codimension $2$ space of type $C_3$ (since through each codimension $2$ space of type $C_3$ there passes a hyperplane of type $H_2$, which follows from $m_{3}^{1}=0$). 
Furthermore, since the number of $H_2$ hyperplanes containing a given codimension $2$ space of type $C_1$ equals the number of $H_1$ hyperplanes containing it, $\gamma$ cannot be contained in a codimension $2$ space $\alpha$ of type 
$C_1$ such that $\gamma\subset\alpha\subset H$. 
Hence in $H$ all codimension $2$ spaces containing $\gamma$ must be of type $C_2$. By the previous lemma we obtain that $X=(q+1)q^{n-2}+\frac{(q^{n-1}+1)(q^{n-2}-1)}{q-1}$.\qed
\end{proof}

\begin{corollary}\label{hypervlak-hyp}
Every hyperplane of type $H_1$ intersects $K$ in the point set of a non-singular hyperbolic quadric.
\end{corollary}
\begin{proof}
By the previous lemma the conditions of Theorem \ref{hyperbolic} are satisfied in each hyperplane of type $H_1$, whenever $n>2$. If $n=2$, the corollary follows by the remark after Theorem \ref{Main}.\qed
\end{proof}

\begin{lemma}\label{punt-hyperbolisch}
Every point of $K$ is contained in at least one hyperplane of type $H_1$.
\end{lemma}
\begin{proof}
Let $p$ be any point of $K$. Assume by way of contradiction that $p$ is contained only in hyperplanes of type $H_2$ and $H_3$. Then, by counting pairs $(r,H)$, $r\in K$, $r\neq p$, $p,r\in H$, $H$ a hyperplane, we obtain
\[ l_2(H_2-1)+\left(\frac{q^{2n}-1}{q-1}-l_2\right)(H_3-1)=\left(\frac{q^{2n}-1}{q-1}-1\right)\frac{q^{2n-1}-1}{q-1}, \]
with $l_2$ the number of hyperplanes of type $H_2$ containing $p$. It follows that $l_2<0$, an absurdity.\qed
\end{proof}

\begin{theorem}\label{theoremparabolic}
The set $K$ is the point set of a non-singular parabolic quadric.
\end{theorem}
\begin{proof}
Let $L$ be any line of $\mathrm{PG}(2n,q)$, and suppose that $\left|L\cap K\right|=x$, with $x\geq3$. Suppose there would be no hyperplane of type $H_1$ containing $L$. Then, if $n_2$ would be the number of hyperplanes of type $H_2$ containing $L$, 
we obtain by counting pairs $(r,H)$, $r\in K$, $r\notin L$, $H$ a hyperplane containing $r$ and $L$
\[ n_2(H_2-x)+\left(\frac{q^{2n-1}-1}{q-1}-n_2\right)(H_3-x)=\left(\frac{q^{2n}-1}{q-1}-x\right)\frac{q^{2n-2}-1}{q-1}. \]
This implies that $n_2<0$, a contradiction. Hence $L$ is contained in a hyperplane of type $H_1$. Corollary \ref{hypervlak-hyp} implies that $x=q+1$. Consequently $K$ is a $(0,1,2,q+1)$-set in $\mathrm{PG}(2n,q)$. By combining 
Lemma \ref{punt-hyperbolisch} and Corollary \ref{hypervlak-hyp} we also see that each point of $K$ is contained in a line $M$ such that $\left|M\cap K\right|=2$. Hence $K$ is non-singular. It follows that $K$ satisfies the conditions of Theorem \ref{quadric}. 
As $\left|K\right|=\frac{q^{2n}-1}{q-1}$ the theorem follows. \qed
\end{proof}

The Main Theorem \ref{Main} is now an immediate consequence of Theorems \ref{theoremHermitian}, \ref{hyperbolic}, \ref{elliptic} and \ref{theoremparabolic}.
\bigskip

{\bf Acknowledgements.} Part of this article was written while the authors were enjoying the warm hospitality of G. Lunardon at the Universita degli Studi di Napoli Federico II. 
The authors also want to thank J.A. Thas for careful proofreading. The first author is a Postdoctoral Fellow of the Science Foundation Flanders(FWO-Vlaanderen). The research of the second author takes place within the project "Linear codes and cryptography" of the Fund for Scientific Research Flanders (FWO-Vlaanderen) (Project nr. G.0317.06) and is supported by the Interuniversitary Attraction Poles Programme-Belgian State-Belgian Science Policy: project P6/26-Bcrypt.

\end{document}